\documentclass[11pt,a4paper, twoside]{amsart}
\usepackage[latin1]{inputenc}

\usepackage{tikz}

\usepackage{amsmath}
\usepackage{amsfonts}
\usepackage{amssymb}
\usepackage{latexsym}

\usepackage{hyperref}
\hypersetup{backref, pdfpagemode=FullScreen, colorlinks=true,
  citecolor=magenta, linkcolor=cyan, urlcolor=blue}

\usepackage{mathtools}
\mathtoolsset{showonlyrefs}

\usepackage{graphicx}
\usepackage{xcolor}
\usepackage{graphicx} 

\usepackage{version} 

\theoremstyle{change}
\newtheorem{theo}{Theorem}[section]
\newtheorem*{theo*}{Theorem}
\newtheorem{lemma}[theo]{Lemma}

\newtheorem{prop}[theo]{Proposition}

\newtheorem{rem}[theo]{Remark}

\def\A{\mathcal{A}}
\def\vf{\mathbf{f}}
\newcommand{\vol}{\mathrm{vol}\,}
\newcommand{\scal}[1]{\langle#1\rangle}
\newcommand{\bnorm}[1]{\left\|#1\right\|}
\newcommand{\con}{\mathrm{cone}\,}
\newcommand{\Ksn}{{\mathcal K}_{(s)}^n} 
\newcommand{\Knull}{{\mathcal K}_{(o)}^n} 
\newcommand{\dual}[1]{{#1}^\star}
\newcommand{\R}{\mathbb{R}} 
\newcommand{\Z}{\mathbb{Z}} 
\newcommand{\cs}{\mathrm{cs}\,}
\newcommand{\ip}[2]{\left\langle #1,#2\right\rangle}
\newcommand{\Kn}{{\mathcal K}^n} 
\newcommand{\va}{{\boldsymbol a}}

\newcommand{\ve}{{\boldsymbol e}}

\newcommand{\vu}{{\boldsymbol u}}
\newcommand{\vv}{{\boldsymbol v}}
\newcommand{\vw}{{\boldsymbol w}}
\newcommand{\vx}{{\boldsymbol x}}
\newcommand{\vy}{{\boldsymbol y}}
\newcommand{\vz}{{\boldsymbol z}}

\newcommand{\vone}{{\boldsymbol 1}}

\newcommand{\vnull}{{\boldsymbol 0}}

\newcommand{\conv}{\mathrm{conv}\,} 
\newcommand{\suk}{\mathrm{h}}
\DeclareMathOperator{\lin}{lin}
\newcommand{\inte}{\mathrm{int}\,}  
\newcommand{\ov}{\overline}

\numberwithin{equation}{section}

\begin{document}
	
\title[successive minima-type inequalities for the polar]{On successive minima-type inequalities for the polar of a
  convex body}
\author{Martin Henk}\author{Fei Xue}
\address{Institut f\"ur Mathematik, Technische Universit\"at Berlin,
	Sekr. Ma 4-1, Strasse des 17 Juni 136, D-10623 Berlin, Germany}
\email{\{henk,xue\}@math.tu-berlin.de}
\thanks{The research of the second named author was supported by a PhD scholarship of the Berlin
	Mathematical School}
	
\begin{abstract}
Motivated by conjectures  of Mahler and Makai Jr., we study bounds on the
volume of a convex body in terms of the successive minima of its polar
body.
\end{abstract}

\maketitle

\section{Introduction}
Let $\Kn$ be the set of all convex bodies, i.e., compact convex sets, in
the $n$-dimensional Euclidean space $\R^n$ with non-empty interior. Let
$\ip{\,\cdot}{\cdot}$ and $\|\cdot\|$  be the standard inner product and
the Euclidean norm in $\R^n$, respectively. We denote by $\Knull\subset \Kn$
the set of all  convex bodies, having the origin as an interior point,
i.e., $\vnull\in\inte(K)$, and by $\Ksn\subset\Knull$ those bodies
which are symmetric with respect to $\vnull$, i.e., $K=-K$. 
The volume of a set $S\subset\R^n$ is its
$n$-dimensional Lebesgue measure and it is denoted by $\vol(S)$.  

For $K\in\Knull$ and $1\leq i\leq n$ let 
\begin{equation*}
  \lambda_i(K)=\min \left\{\lambda>0 : \dim(\lambda\,K\cap\Z^n)\geq i\right\}
\end{equation*}
be its $i$th successive minimum, which is the smallest positive dilation factor
$\lambda$ such that $\lambda\,K$ contains $i$ linearly independent lattice points
of the lattice $\Z^n$. The so-called second theorem of Minkowski on successive
minima provides optimal upper and lower bounds on the volume of a
symmetric 
convex body $K\in\Ksn$  in terms  of its successive
minima. These bounds can be easily generalized to the class
$K\in\Kn$ as follows  
\begin{equation}
 \label{eq:minkowski}
    \frac{2^n}{n!}\prod_{i=1}^{n}\frac{1}{\lambda_i(\cs({K}))}\leq\vol(K)\leq
     2^n\prod_{i=1}^{n}\frac{1}{\lambda_i(\cs({K}))}.
\end{equation}
where $\cs(K)=\frac{1}{2}(K-K)\in\Ksn$ is the central symmetral  of $K$. The
$n$-dimensional unit cube $C_n$ shows that the upper bound is optimal, and 
its polar body $\dual{C_n}$, the $n$-dimensional cross-polytope,
attains the lower bound. Here, in general, the polar body of $K\in\Kn$
is defined as
\begin{equation*}
  \dual{K} =\left\{\vy\in\R^n : \ip{\vx}{\vy}\leq 1\text{ for all
    }\vx\in K\right\}.
\end{equation*}
Mahler \cite{Mah39-2} studied for $K\in\Ksn$ the volume product
$M(K)=\vol(K)\vol(\dual{K})$ and conjectured 
\begin{equation}
   M(K)\geq M(C_n)= \frac{4^n}{n!}.
\label{conj:mahler}
\end{equation} 
Mahler \cite{Mah39} verified the conjecture in dimension 2, and there was a recent announcement of
its proof in dimension 3 by \cite{IS17}.  In the general case, it is
conjectured that for $K\in\Kn$  
\begin{equation}
   M(K)\geq M(S_n)= \frac{(n+1)^{n+1}}{(n!)^2},
\label{conj:mahler2}
\end{equation} 
where $S_n$ is a simplex with centroid at the origin. This is only
known to be true in the plane \cite{Mah39}.

Combining the upper bound in \eqref{eq:minkowski}  with the conjectured
lower bound \eqref{conj:mahler} leads for $K\in\Ksn$ to the inequality 
\begin{equation}
   \vol(K) \geq \frac{2^n}{n!}\prod_{i=1}^n \lambda_i(\dual{K}). 
\label{conj:mahler_2} 
\end{equation}
 This inequality, which would be best possible, for instance, for the
 cross-polytope  $\dual{C_n}$, was also conjectured by  Mahler \cite{Mah74}, and the
 previous mentioned results on the volume product $M(K)$ implies that it is
 true for $n=2$ and (probably for $n=3$). Even the weaker
 inequality, 
\begin{equation}
   \vol(K) \geq \frac{2^n}{n!} \lambda_1(\dual{K})^n,  
\label{conj:mahler_2_1} 
\end{equation}
which has also been studied by Mahler,  is open for $n\geq 4$. 

For not necessarily symmetric bodies the same problem was studied by
Makai Jr., and he conjectured for $K\in \Kn$ 
\begin{equation}
   \vol(K) \geq \frac{n+1}{n!} \lambda_1(\dual{\cs(K)})^n,  
\label{conj:makai} 
\end{equation}
and proved it for $n=2$ (\cite{FeM74, Mak78}). In view of \eqref{conj:mahler_2_1}, one might
conjecture  the stronger inequality 
\begin{equation}
   \vol(K) \geq \frac{n+1}{n!} \prod_{i=1}^n\lambda_i(\dual{\cs(K)})^n,  
\label{conj:makai_2} 
\end{equation}
which would be possible as the simplex
$S_n=\conv\{\ve_1,\dots,\ve_n,-\vone\}$ shows, where $\ve_i$ is the
$i$th   unit vector and $\vone$ is the all
1-vector. For $n=2$ this is an immediate consequence of the upper bound in
\eqref{eq:minkowski} and Eggelston \cite{Eg61} inequality for planar convex
bodies  
\begin{equation}
  \vol(K)\vol(\dual{\cs(K)})\geq 6.
  \label{eq:eggelston} 
\end{equation}
Actually, we believe that taking into account all successive minima one should get
a stronger lower bound as in \eqref{conj:makai_2}, and here we show in
the plane.
\begin{theo} Let $K\in\mathcal{K}^2$. Then 
\begin{equation}
\begin{split}
        \vol(K) & \geq \frac{3}{2} \lambda_1(\dual{\cs(K)})
        \lambda_2(\dual{\cs(K)})\\
        & +\frac{1}{2}\lambda_1(\dual{\cs(K)})\Big(\lambda_2(\dual{\cs(K)})-\lambda_1(\dual{\cs(K)})\Big),
\end{split}
%
\label{eq:mainineq}
\end{equation}  
and  equality holds if and only if $K$ is up to translations and
unimodular transformations equal to the triangle $T_{s,t}=\conv\{(-s,t-s), (s,t),
(0,t)\}$ with $t\geq s\in\R_{> 0}$.
\label{thm:planar}
\end{theo}

We note that for the triangle $T_{s,t}$ (see Figure \ref{fig:trianglest}) we have
\begin{equation}
  \label{eq:1}
  \lambda_{1}(\dual{\cs(K)})=s, \,\lambda_{2}(\dual{\cs(K)})=t,\,
   \quad\text{ and } \vol(T_{s,t})=2ts-\frac{1}{2}s^2. 
\end{equation}

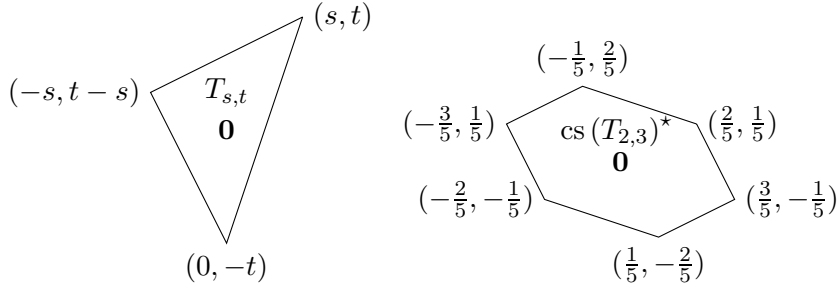
\begin{figure}[htb]
\begin{center}
	\begin{tikzpicture}[scale=0.5]
	\draw (2,3) -- (-2,1) -- (0,-3) -- (2,3);
	\draw (2,3) node[right] {$(s,t)$};
	\draw (-2,1) node[left] {$(-s,t-s)$};
	\draw (0,-3) node[below] {$(0,-t)$};
	\draw (0,0) node {$\vnull$};
        \draw (0,1) node {$T_{s,t}$};
	\end{tikzpicture}
	\begin{tikzpicture}[scale=0.5]
	\draw (2,1) -- (-1,2) -- (-3,1) -- (-2,-1) -- (1,-2) -- (3,-1) -- (2,1);
	\draw (2,1) node[right] {$(\frac{2}{5},\frac{1}{5})$};
	\draw (-1,2) node[above] {$(-\frac{1}{5},\frac{2}{5})$};
	\draw (-3,1) node[left] {$(-\frac{3}{5},\frac{1}{5})$};
	\draw (-2,-1) node[left] {$(-\frac{2}{5},-\frac{1}{5})$};
	\draw (1,-2) node[below] {$(\frac{1}{5},-\frac{2}{5})$};
	\draw (3,-1) node[right] {$(\frac{3}{5},-\frac{1}{5})$};
	\draw (-0.1,0.8) node {$\dual{\cs(T_{2,3})}$};
        \draw (0,0) node {$\vnull$};
	\end{tikzpicture}
\end{center}
\caption{Triangle $T_{s,t}$ attaining equality in Theorem
  \ref{thm:planar}, and $\dual{\cs(T_{2,3})}$ for the
  parameters $s=2, t=3$.}
\label{fig:trianglest}
\end{figure} 


We remark that Makai\&Martini\cite[Proposition 3.1]{MM16} (see also
Makai\cite[Proposition 1]{Mak78})
verified for simplices  $S\in\Kn$ the
conjectured higher-dimensional analogue of \eqref{eq:eggelston},
namely
\begin{equation*}
   \vol(S)\vol(\dual{\cs(S)})\geq 2^n\frac{n+1}{n!}.
\end{equation*}
Application of Minkowski's upper bound \eqref{eq:minkowski} shows inequality
\eqref{conj:makai_2} for simplices.  For arbitrary convex
bodies $K\in\Kn$  one may write (cf.~\cite{MM16})
\begin{equation*}
   \vol(K)\vol(\dual{\cs(K)})=\frac{\vol(K)}{\vol(\cs(K))} M(\cs(K))
   \geq \frac{2^n}{\binom{2n}{n}}\frac{\pi^n}{n!} \geq \frac{(\pi/2)^n}{n!},
\end{equation*}
where the lower bound on the  volume product is Kuperberg's
bound\cite{Ku08}, and the lower bound on  the ratio
$\frac{\vol(K)}{\vol(\cs(K))}$ is the Rogers-Shephard bound (cf.,
e.g., \cite[Theorem 10.4.1]{Schneider:2014}. Hence, in
general, we have the bound 
\begin{equation}
   \vol(K) \geq \frac{(\pi/4)^n}{n!} \prod_{i=1}^n\lambda_i(\dual{\cs(K)})^n.  
\end{equation}
In contrast to the lower bounds, in the case of upper bounds we have a
complete picture.

\begin{theo} Let $K\in\Kn$. 
\begin{enumerate} 
\item  Then 
\begin{equation}
  \vol(K)\leq 2^n \prod_{i=1}^n \lambda_i(\dual{\cs(K)}).
\label{eq:uppersym}
\end{equation} 
The inequality is best possible.  
\item If the centroid of $K$ is at the origin, then 
\begin{equation}
  \vol(K)\leq \frac{(n+1)^n}{n!} \prod_{i=1}^n \lambda_i(\dual{K}).
\label{eq:uppernotsym}
\end{equation} 
The inequality is best possible.
\item  For arbitrary $K\in\Knull$, the volume is in general not  
  bounded from above by the product of $\lambda_i(\dual{K})$.
\end{enumerate} 
\label{thm:upperbounds} 
\end{theo} 
Observe, that $\lambda_i(\dual{K})\leq
\lambda_i(\dual{\cs(K)})$, $1\leq i\leq n$, cf.~Proposition \ref{prop:succ}.

Finally, we would like to mention that a weaker inequality
than \eqref{conj:makai} was recently
studied by Alavarez et al.\cite{ABT16}. They conjecture for $K\in\Knull$ 
\begin{equation}
   \vol(K) \geq \frac{n+1}{n!} \lambda_1(\dual{K})^n
\label{conj:alvarez} 
\end{equation} 
with equality if and only if $K$ is a simplex whose vertices are the
only non-trivial lattice points. By the discussion above we know that it is true in the plane, for simplices
and with $(\pi/4)^n/n!$ instead of $(n+1)^n/n!$
(cf.\cite[Theorem II]{ABT16}). Moreover, according to Theorem
\ref{thm:upperbounds} iii), there is no upper bound  on the volume of this
type.
For an optimal lower bound on the volume of centered convex body $K$, i.e., the
centroid if $K$ is at the origin,  in terms of $\lambda_i(K)$ we refer
to \cite{HHC16}.  Instead of extending  Makai's conjecture
\eqref{conj:makai} via higher successive minima
(cf. \eqref{conj:makai}), Gonzales Merion \& Schymura \cite{GSch17}
studied possible extensions via the so called covering minima.


The paper is organized as follows: First, i.e., in Section 2, we will
verify the upper bounds of Theorem \ref{thm:upperbounds}. Then as
preparation for the proof of Theorem \ref{thm:planar} we will study
gauge functions in Section 3.  Finally, 
the content of Section 4 is the proof of Theorem \ref{thm:planar}

For a general background and information on Convex Geometry and
Geometry of Numbers we refer to the books \cite{Gr07, GrL87, Schneider:2014}.

\section{Proof of Theorem \ref{thm:upperbounds}} 
In order to deal with the polar of a convex body $L\in\Kn$, say, it is
convenient to look at its  support function $\suk_L:\R^n\to \R$ given
by 
\begin{equation*}
  \suk_L(\vu)=\max\{\ip{\vu}{\vx} :\vx\in L\} 
\end{equation*}
for $\vu\in\R^n$. Then for $\lambda\in\R_{\geq 0}$ we have 
\begin{equation} 
  \vy\in\lambda\,\dual{L} \text{ if and only if
  }\suk_L(\vy)\leq\lambda.
\label{eq:polarsupp}
\end{equation}
First we observe a simple relation between the successive minim of
$K\in\Knull$ and $\cs(K)$ which, for $i=1$ was already  pointed out by
 by Alvarez et al. \cite{}.   

\begin{prop}  Let $K\in\Knull$. Then, for $1\leq i\leq n$,  
  \begin{equation*}
          \lambda_i(\dual{K})\leq \lambda_i(\dual{\cs(K)}).
  \end{equation*}
\label{prop:succ}
\end{prop} 
\begin{proof} 
Let $\dual{\lambda_i}= \lambda_i(\dual{\cs(K)})$ and
  let  $\vz_1,\dots,\vz_i\in\Z^n$ be
  linearly independent lattice points with 
  $\vz_j\in\dual{\lambda_j}\,\dual{\cs(K)}$, $1\leq j\leq i$.
  Then, for $1\leq j\leq i$,  by the linearity
  of the support function 
  \begin{equation*}
    \dual{\lambda_i}\geq \suk_{\frac{1}{2}(K-K)}(\vz_j) =
    \frac{1}{2}\left(\suk_K(\vz_j)+\suk_K(-\vz_j)\right)
    \geq\min\{\suk_K(\vz_j), \suk_K(-\vz_j)\}.
  \end{equation*}
Hence, either $\vz_j$ or $-\vz_j$ belongs to $\dual{\lambda_i}\dual{K}$
for $1\leq j\leq i$, and thus $\lambda_i(\dual{K})\leq \dual{\lambda_i}=\lambda_i(\dual{\cs(K)})$. 
\end{proof} 

For the proof of Theorem \ref{thm:upperbounds} ii) we will also need a
classical result of Grünbaum \cite{Gruenbaum:1960}, saying that for $K\in\Knull$ and for any
halfspace $H^+=\{\vx\in\R^n : \ip{\va}{\vx}\geq 0\}$ containing the centroid of $K$ it
holds 
\begin{equation}
  \label{eq:gruenbaum}
    \vol(K\cap H^+)\geq \left(\frac{n+1}{n}\right)^n\,\vol(K).
\end{equation}

\begin{proof}[Proof of Theorem \ref{thm:upperbounds}]

For i),  let $\vz_1,\dots,\vz_n\in\Z^n$ be
  linearly independent lattice points with 
  $\vz_i\in \lambda_i(\dual{\cs(K)})\,\dual{\cs(K)}$, $1\leq i\leq n$. Then we
  certainly have 
  \begin{equation*}
     K\subseteq P= \{\vx\in\R^n: -\suk_K(-\vz_i)\leq  \ip{\vz_i}{\vx}\leq\suk_K(\vz_i),\,1\leq
     i\leq n\}. 
  \end{equation*}
In order to estimate the volume of the parallelepiped on the right
hand side we observe, that in view of \eqref{eq:polarsupp},
$2\,\lambda_i(\dual{\cs(K)})=\suk_K(\vz_i)+\suk_K(-\vz_i)$, $1\leq
i\leq n$, and thus 
\begin{equation*}
\vol(K) \leq \vol(P)=
  \frac{1}{|\det(\vz_1,\ldots,\vz_n)|}\prod_{i=1}^n
  2\,\lambda_i(\dual{\cs(K)}) \leq  2^n \prod_{i=1}^n \lambda_i(\dual{\cs(K)}), 
\end{equation*}
where in the last inequality we used $\det(\vz_1,\ldots,\vz_n)\in\Z\setminus\{0\}$.
The cube $C_n$ with its polar body
$\dual{C_n}=\conv\{\pm\ve_1,\dots,\pm\ve_n\}$ shows that the equality is
best possible.

Now assume that the centroid of $K$ is at the origin. Let $\dual{\lambda_i}= \lambda_i(\dual{K})$,
  $1\leq i\leq n$, and let $\vz_1,\dots,\vz_n\in\Z^n$ be
  linearly independent lattice points with 
  $\vz_i\in \dual{\lambda_i}\,\dual{K}$. Then, for $1\leq i\leq n$, 
\begin{equation} 
  \suk_K(\vz_i)\leq  \dual{\lambda_i}.
  \label{eq:succdual}
\end{equation} 
Moreover, we consider the halfspace  
\begin{equation*}
  H^+=\{\vx\in\R^n: \ip{\frac{1}{\dual\lambda_1}\vz_1+\cdots +\frac{1}{\dual\lambda_n}\vz_n}{\vx}\geq 0\}.  
\end{equation*}  
Then we conclude from \eqref{eq:succdual}
\begin{equation}
  K\cap H^+\subseteq S=\{\vx \in \R^n : \ip{\vz_i}{\vx}\leq
  \dual{\lambda_i},\,1\leq i\leq n\}\cap H^+.
  \label{eq:4} 
\end{equation}
In order to calculate the volume of the simplex $S$ we observe that it
is the image of the simplex
\begin{equation*}
  \ov S=\{\vx\in\R^n : \ip{\ve_i}{\vx}\leq 1,\,1\leq i\leq n,\,\ip{\vone}{\vx}\geq 0 \} 
\end{equation*}
with respect to the linear map
$A=(\frac{1}{\dual\lambda_1}\vz_1,\dots,
\frac{1}{\dual\lambda_n}\vz_n)$.  
Hence, 
\begin{equation}
       \vol(S)=\frac{1}{|\det(\vz_1,\ldots,\vz_n)|}\prod_{i=1}^n
       \dual{\lambda_i}\vol(\ov S)=\frac{1}{|\det(\vz_1,\ldots,\vz_n)|}\prod_{i=1}^n
       \dual{\lambda_i}\frac{n^n}{n!},  
     \end{equation}
and together with Grünbaum's bound \eqref{eq:gruenbaum} and
\eqref{eq:4} we conclude
\begin{equation*}
  \begin{split} 
  \vol(K) & \leq \left(\frac{n+1}{n}\right)^n\vol(K\cap H^+)\\  &\leq
  \left(\frac{n+1}{n}\right)^n\vol(S) = 
  \frac{1}{|\det(\vz_1,\ldots,\vz_n)|}\frac{(n+1)^n}{n!} \prod_{i=1}^n
  \dual{\lambda_i}.
  \end{split} 
\end{equation*}
Again, since $\det(\vz_1,\ldots,\vz_n)\in\Z\setminus\{0\}$ we get
the desired bound. 
The simplex $T_n=\{\vx\in\R^n : \ip{\ve_i}{\vx}\leq 1,\,1\leq
i\leq n, \ip{\vone}{\vx}\geq -1\}$ with volume $(n+1)^n/n!$ and
$\dual{T_n}=\conv\{\ve_1,\dots,\ve_n,-\vone\}$ shows the bound is best
possible.

Finally, we point out that the assumption on the centroid  is crucial
for ii). To this end,  for $s\geq 1$ we consider the simplices 
$T(s)=\{\vx\in\R^n : \ip{\ve_i}{\vx}\leq 1,\,1\leq
i\leq n, \ip{\frac{1}{s}\vone}{\vx}\geq -1\}$. Then $\dual{T(s)} = \conv\{-\frac{1}{s}\vone,
\ve_1,\dots,\ve_n\}$ and thus $\lambda_i(\dual{T(s)})=1$, $1\leq i\leq
n$. On the other hand, $\vol(T(s))\to\infty$ as $s$ approaches
$\infty$. This verifies iii).
\end{proof}

\section{Gauge function}
Here we collect some basic facts about gauge functions $\bnorm{\vx}_{K}$ associated to a
$K\in\Knull$ which are defined by 
\begin{equation*}
\bnorm{\vx}_{K}:\R^n\rightarrow[0,\infty)
\end{equation*}
defined by
\begin{equation*}
\bnorm{\vx}_K=\min\left\{t\geq 0 : \vx\in t\,K\right\}.
\end{equation*}
As it is well known, $\bnorm{\cdot}_K$ satisfies the following properties:
\begin{enumerate}
	\item $\bnorm{\vx}_K\geq 0$ with equality if and only if
          $\vx=\vnull$, 
	\item $\bnorm{\lambda\,\vx}_K=\lambda\,\bnorm{\vx}_K$ for
          $\lambda\in\R_{\geq 0}$,
	\item $\bnorm{\vx+\vy}_K\leq \bnorm{\vx}_K+\bnorm{\vy}_K$. 
\end{enumerate}
Conversely, if $\bnorm{\cdot}$ is a function satisfying these three
properties, then its unit ball $B=\{\vx\in\R^n:\bnorm{\vx}\leq 1\}$ 
is a convex body in $\Knull$ and $\bnorm{\cdot}=\bnorm{\cdot}_B$. 

We also note that if $T:\R^n\rightarrow\R^n$ is an invertible linear transformation, then
$\bnorm{\vx}_{T(B)}=\bnorm{T^{-1}\vx}_B$ for all $\vx\in\R^n$. From the
definition of the gauge function it is evident that for $K\in\Knull$ 
\begin{equation}
           \bnorm{\vx}_K=\suk_{\dual{K}}(\vx).
  \label{eq:supportgauge}  
\end{equation}
Hence, from the linearity of the support function we immediately
obtain  
\begin{equation}
\begin{split} 
\bnorm{\vx}_{\dual{\cs(K)}}&=\suk_{\cs(K)}(\vx)\\ &= 
    \frac{1}{2}\left(\suk_{K}(\vx)+\suk_{K}(-\vx)  \right)= 
 \frac{1}{2}\left(\bnorm{\vx}_{\dual{K}}+\bnorm{-\vx}_{\dual{K}}\right).
\end{split} 
\label{eq:cs}  
\end{equation}

Combining this with the triangle inequality we conclude for $K\in\Knull$ 
\begin{equation}
\begin{split} 
	& \bnorm{\vx+\vy}_{\dual{\cs(K)}}=\bnorm{\vx}_{\dual{\cs(K)}}+\bnorm{\vy}_{\dual{\cs(K)}}
       \text{ if and only if } \\ 
	& \bnorm{\vx+\vy}_{\dual{K}}=\bnorm{\vx}_{\dual{K}}+\bnorm{\vy}_{\dual{K}} 
	\text{ and } 
	 \bnorm{-(\vx+\vy)}_{\dual{K}}=\bnorm{-\vx}_{\dual{K}}+\bnorm{-\vy}_{\dual{K}}.
\end{split} 
 \label{eq:equalitygauge}
\end{equation}

\section{Proof of Theorem \ref{thm:planar}}
Since the inequality of Theorem \ref{thm:planar}, i.e.,  
\begin{equation*}
\begin{split} 
\vol(K)&\geq\frac{3}{2} \lambda_1(\dual{\cs(K)})
        \lambda_2(\dual{\cs(K)})\\
        &
        +\frac{1}{2}\lambda_1(\dual{\cs(K)})\Big(\lambda_2(\dual{\cs(K)})-\lambda_1(\dual{\cs(K)})\Big)\\ 
   &= 2 \lambda_1(\dual{\cs(K)})\lambda_2(\dual{\cs(K)})-\frac{1}{2}\lambda_1(\dual{\cs(K)})^2
\end{split} 
\end{equation*} 
is invariant with respect to translations and unimodular
transformations, we may assume  
that $K\in \Knull$, $\lambda_2(\dual{\cs(K)})=1$ and 
the successive minima $\lambda_i(\dual{\cs(K)})$ are obtained in
direction of the unit vectors, i.e., $\ve_i(\dual{\cs(K)})\in
\lambda_i(\dual{\cs(K)}) \dual{\cs(K)}$, $i=1,2$. The latter is due to
the fact that in the plane we can always find
$\vz_i\in\lambda_i(\dual{\cs(K)}) \dual{\cs(K)}\cap\Z^2$ building a basis of
$\Z^2$ \cite[Theorem 4, p.20]{GrL87}.

Hence, for a
fixed $t\geq 1$ we are interested in the minimal volume among all
convex bodies in the set  
\begin{equation*}
\begin{split}
\A(t)=\Big\{K\in\mathcal{K}_{(o)}^2 :\, &
\lambda_1(\dual{\cs(K)})=\frac{1}{t},\lambda_2(\dual{\cs(K)})=1,\\ 
&\ve_i\in\lambda_i(\dual{\cs(K)})\,\dual{\cs(K)}\cap\Z^2,i=1,2\Big\}.
\end{split}
\end{equation*}
Observe, that all bodies in $\A(t)$ are contained in the rectangle 
$[-1/t,1/t]\times [-1,1]$ and since the volume of all these bodies is
lower bounded by $3/2\cdot 1/t$ (cf.~\eqref{conj:makai_2}, which is
true for $n=2$), Blaschke's selection
theorem (cf., e.g., \cite[Theorem 6.3]{Gr07} ensures the existence of a convex bodies  in $\A(t)$
having minimal positive volume. We denote these bodies by $\mathcal{M}(t)$,
i.e., 
\begin{equation*}
   \mathcal{M}(t)=\left\{M\in \A(t): \vol(M)=\min\{\vol(K): K\in\A(t)\}\right\}.  
\end{equation*}
Observe, that due to the triangle (cf.~Theorem \ref{thm:planar}) 
\begin{equation*}
  T_{1,1/t} =\conv\{(-1/t,1-1/t), (1/t,1), (0,-1)\},
\end{equation*}
 we know that for
$K\in \mathcal{M}(t)$ (cf.~\eqref{eq:1}) 
\begin{equation}
    \vol(K)\leq \vol(T_{1,1/t} ) =2\,\frac {1}{t}-\frac{1}{2}\frac{1}{t^2}
\label{eq:upmin}
\end{equation}
and Theorem \ref{thm:planar} claims that this is indeed the minimum.

In the following we will prove different geometric properties of
bodies $S \in \mathcal{M}(t)$
(or better of $\dual{S}$)  and
at the end in Proposition \ref{prop:min} we conclude that
$\mathcal{M}(t)$ contains only ---  up to
translations and unimodular transformations --- the triangle
$T_{1,1/t}$. This proves Theorem \ref{thm:planar}.

Due to the definition of the successive minima,  all the lattice
points  of $\dual{\cs(K)}$ for $K\in\A(t)$ 
are either contained in the boundary of $\dual{\cs(K)}$ or lie on the
line  $\lin\{\ve_1\}$. For such a $K\in \mathcal{A}(t)$ we set   
\begin{equation*}
\begin{split}
     C_o(K) & =\{\vz\in\Z^2 : \bnorm{\vz}_{\dual{\cs(K)}}=1\}\cup\{\pm
     \ve_1\},\\ 
     C(K) & = \Big\{ \vz/\bnorm{\vz}_{\dual{K}}: \vz\in C_o(K)\Big\}.
\end{split} 
\end{equation*}

The points in  $C(K)$ are our main objective by which we will show
geometric properties of bodies in $\mathcal{M}(t)$. 

\begin{prop} Let $K\in \mathcal{M}(t)$. Then $K$ is a polygon and the
    relative interior of each edge of $\dual{K}$ contains a point of
    $C(K)$.  
\label{prop:one}
\end{prop} 

\begin{proof}  
First, we prove that $\dual{K}$ and thus $K$ is a polygon. Since
$\dual{\cs(K)}, \dual{K}$ are bounded, both are strictly contained in a square
$C_N=[-N,N]^2$ for some large $N\in\R_{>0}$. 
For any non-zero lattice point $\vz\in C_N$, there is a supporting
hyperplane in the boundary point
$\frac{\vz}{\bnorm{\vz}_{\dual{K}}}$ with respect to $\dual{K}$. 
Let $C$ be the intersection of the corresponding halfspaces 
containing $\dual{K}$ together with the halfspaces bounding $C_N$. 

Obviously, $\dual{C}\subseteq K$ is a
 polygon and we claim that $\dual{C}\in\A(t)$. In order to avoid
 confusion, we set $P=\dual{C}$ and so $C=\dual{P}$ and we want to show
$P\in \A(t)$.  
 To this end, we observe
 that for all  $\vz\in C_N$ we have by construction
 \begin{equation*}
     \bnorm{\vz}_{\dual{P}}= \bnorm{\vz}_{\dual{K}} 
 \end{equation*}
and hence, in view of \eqref{eq:cs} 
 \begin{equation*}
     \bnorm{\vz}_{\dual{\cs(P)}}= \bnorm{\vz}_{\dual{\cs(K)}}. 
 \end{equation*}
For  $\vz\in \Z^2\setminus C_N$ we know by construction 
\begin{equation*}
         \bnorm{\vz}_{\dual{P}}\geq\bnorm{\vz}_{\dual{K}}
\end{equation*}
and so 
\begin{equation*}
     \bnorm{\vz}_{\dual{\cs(P)}}\geq\bnorm{\vz}_{\dual{\cs(K)}}>1. 
 \end{equation*} 
Hence, $P\in\A(t)$, $P\subseteq K$  and since $K\in\mathcal{M}(t)$, we
 must have $K=P$.

Next assume that there is an edge of
$\dual{K}$ which does not contain in its relative
interior a point of $C(K)$. Then  we may move the edge  a bit outward so that for this new polygon
$\dual{K_\epsilon}$, considered as the polar of a polygon
$K_\epsilon$,  it holds
\begin{equation*}
   \bnorm{\vz}_{\dual{K_\epsilon}}= \bnorm{\vz}_{\dual{K}} \text{ and thus }
 \bnorm{\vz}_{\dual{\cs(K_\epsilon)}}= \bnorm{\vz}_{\dual{\cs(K)}}
\end{equation*}
for all $\vz\in C(K)$. For all other lattice points $\vz$ (which are
not contained in $\lin\{\ve_1\}$),  we know
$\bnorm{\vz}_{\dual{\cs(K)}}>1$ and hence, by moving just a little bit
we still have $\bnorm{\vz}_{\dual{\cs(K_\epsilon)}}>1$ for
these points.

Thus $K_\epsilon\in\A(t)$ but $K_\epsilon$ is strictly
contained in $K$, contradicting its minimality with respect to the volume. 
\end{proof} 

In order to give a bound on the size of $C(K)$, $K\in\mathcal{M}(t)$,
we need the next lemma. 

\begin{lemma}
	Let $K\in\mathcal{M}(t)$, and let $(m,n)\in C_0(K)$. Then $n\in\{-1,0,1\}$.
	\label{lem:line}
\end{lemma}

\begin{proof}
	Assume that $(m,n)\in C_0(K)$ with $n\geq 2$, which is trivially a primitive lattice point.
	
	Since $(0,t),(0,-t)\in \dual{\cs(K)}$ and $(m,n)\in \dual{\cs(K)}$, the intersection of 
	\begin{equation*}
	\conv\{(0,t),(0,-t),(m,n)\}
	\end{equation*}
	with  the line $\{\vx\in \R^2 : x_2=1\}$  has length greater
        or equal than $t\geq 1$.  If the length is strictly greater
        than 1, this  intersection contains a lattice point
        $\vv\in\Z^2$ with $\bnorm{\vv}_{\dual{\cs(K)}}<1$. 

The only
        remaining case is $n=2$ and $t=1$, and since then $\ve_1,\ve_2$ are in the
        boundary, $\dual{\cs(K)}=\conv\{\pm\ve_1, \pm (1,2)\}$. Hence
        up to translations we can assume that $K$ is the  parallelogram
        $\conv\{\pm\ve_1, \pm (1,-1)\}$ of volume $2$. Hence, $K\notin
        \mathcal{M}(t)$ (cf.\eqref{eq:upmin}).   
\end{proof}

\begin{rem} Let $K\in\mathcal{M}(t)$. By Lemma \ref{lem:line} we get
\begin{enumerate}
\item if $|C(K)|=4$ then $C_0(K)=\{\pm\ve_1,\pm\ve_2\}$, 
\item  if $|C(K)|=6$, then
	\begin{equation*}
	C_0(K)=\{\pm\ve_1,\pm\ve_2,\pm(\ve_1+\ve_2)\}\text{ or }\{\pm\ve_1,\pm\ve_2,\pm(\ve_2-\ve_1)\}. 
      \end{equation*}
      Observe that both configurations are unimodular equivalent.
\end{enumerate} 
\label{rem:six} 
\end{rem} 

Next we show that $C(K)$ can not have more than $6$ points. 

\begin{prop} 
	Let $K\in\mathcal{M}(t)$. Then $|C(K)|\leq 6$, i.e., $|C(K)|=4$ or $6$.
\label{prop:six}
\end{prop}
\begin{proof}
	Let $K\in\mathcal{M}(t)$ and assume  $|C(K)|>6$. Then in view of Lemma
        \ref{lem:line}  
there are at least three points in $C_0(K)$
with last coordinate $1$, and at least three points with last coordinate
$-1$. All these points lie in the boundary of $\dual{\cs(K)}$ and hence, $\dual{\cs(K)}$ has an edge contained in
the line $\{\vx:x_2=1\}$ and one contained in $\{\vx:x_2=-1\}$. Hence,
$\cs(K)$ has the  vertices $\pm\ve_2$, which shows that $K$ has two
vertices $\vx,\vy$ with $\vx-\vy=2\ve_2$.

On the other hand we have  
$\bnorm{\ve_1}_{\dual{\cs(K)}}=\frac{1}{t}$ and thus
$\suk_{\cs(K)}(\ve_1)=\frac{1}{t}$. Hence,  $K$ contains also two
vertices differing in the first coordinate by $\frac{2}{t}$.
Altogether, this
shows that the volume of $K$ is at least $2/t$ and hence,  $K\notin
        \mathcal{M}(t)$ (cf.\eqref{eq:upmin}).   
	
\end{proof}

Now we study the number of points of $C(K)$ in each edge of
$\dual{K}$. The following lemma shows that, under some translation of
$K$, the relation between the points of $C(K)$ and the edges of $\dual{K}$ does not change.

\begin{lemma}
	Let $K\in\mathcal{K}_{(o)}^2$ and $-\vu\in\inte(K)$. Let
        $\vv\in\R^2$, such that $\frac{\vv}{\bnorm{\vv}_{\dual{K}}}$
        lies in the relative interior of the edge $E=\{\vx\in K:\ip{\vx}{\vf}=1\}$ of
        $\dual{K}$. 

   Then $\frac{\vv}{\bnorm{\vv}_{\dual{(K+\vu)}}}$  lies in the
   relative interior of the edge $E'=\{\vx\in \dual{(K+\vu)} :\ip{\vx}{\vf+\vu}=1\}$ of $\dual{(K+\vu)}$.
	\label{lem:translate}
\end{lemma}

\begin{proof}
        By assumption $\vf$ is a vertex of $K$ and so is $\vf+\vu$ a
        vertex of
        $K+\vu$. Hence $E'$ is an edge of $\dual{(K+\vu)}$.    
	Next, since $\ip{\vv}{\vf}=\bnorm{\vv}_{\dual{K}}$ and 
	\begin{equation*}
	\bnorm{\vv}_{\dual{(K+\vu)}}=\suk_{K+\vu}(\vv)=\suk_{K}(\vv)+\ip{\vv}{\vu}=\bnorm{\vv}_{\dual{K}}+\ip{\vv}{\vu},
	\end{equation*}
	we find 
	\begin{equation*}
	\ip{\vv}{\vf+\vu}=\ip{\vv}{\vf}+\ip{\vv}{\vu}=\bnorm{\vv}_{\dual{K}}+\ip{\vv}{\vu}=\bnorm{\vv}_{\dual{(K+\vu)}}.
	\end{equation*}
	Thus $\frac{\vv}{\bnorm{\vv}_{\dual{(K+\vu)}}}\in E'$, and
        since $\vv/\bnorm{\vv}_{\dual{K}}$ was only contained in the
        edge $E$, $\frac{\vv}{\bnorm{\vv}_{\dual{(K+\vu)}}}$ also
        belongs to the relative interior of $E'$.
\end{proof}

Next we describe in more detail the relation of the points of $C(K)$ 
and the edges of $\dual{K}$.

\begin{prop} 
 $\mathcal{M}(t)$ contains a polygon  $K$ such that the relative
 interior  of  each edge of $\dual{K}$ contains  

\begin{equation} 
\begin{split} 
 i) & \text{ at least two points of $C(K)$, or }\\
ii) & \text{ one point of $C(K)$, while }
\frac{\ve_1}{\bnorm{\ve_1}_{\dual{K}}} \text{ or } \frac{-\ve_1}{\bnorm{-\ve_1}_{\dual{K}}}\\ 
& \text{ is a vertex of this edge.}
\end{split} 
\tag{P}
\label{eq:property}
\end{equation} 
Moreover, each $K\in\mathcal{M}(t)$ has at most 4 edges, and  if $K\in
\mathcal{M}(t)$ is a triangle, then $K$ satisfies property \eqref{eq:property}.
	\label{prop:a2}
\end{prop}

\begin{proof}   In the following we show that for each
  $K\in\mathcal{M}(t)$  there exists another polygon
  $K'\in\mathcal{M}(t)$ with the same number of edges as $K$
  satisfying property \eqref{eq:property}.  
  Together with Proposition \ref{prop:six} this implies that each
  $K\in\mathcal{M}(t)$ has at most 4 edges. 

  So let $K\in\mathcal{M}(t)$ be a polygon which does not fullfil
  \eqref{eq:property}. 
Then, in view Proposition \ref{prop:one} we may
  assume  that  $\dual{K}$ has an edge $E$   
\begin{equation*}
  E =\left\{\vx\in\R^2: \scal{\vf,\vx}=1\right\}\cap\dual{K}  
\end{equation*}
containing only one point $\vu=(x_0,y_0)\in C(K)$ in its relative
 interior, and  $\frac{\pm \ve_1}{\bnorm{\pm \ve_1}_{\dual{K}}}$ is  not a vertex of $E$. 

Let $\{\vf,\vf_1,\dots,\vf_k\}$ be the vertices of $K$, and
$E_i=\left\{\vx\in\R^2: \scal{\vf_i,\vx}=1\right\}$, $1\leq i\leq k$,
be the  supporting lines of the other edges of $\dual{K}$. For the
lines $E$, $E_i$ we denote by $\ov E$, $\ov E_i$ the corresponding
halfspaces containing $\dual{K}$, i.e.,
\begin{equation*}
   \dual{K}=\ov E\cap \bigcap_{i=1}^k \ov E_i.
\end{equation*}

Let us parametrize $E$ by the angle 
$\theta_0\in[0,2\pi)$ such that 
\begin{equation*}
   E=\{(x,y)\in\R^2
   :(\cos\theta_0)(x-x_0)+(\sin\theta_0)(y-y_0)=0\}. 
\end{equation*}

Then for a  small $\epsilon>0$ and $\theta\in
(\theta_0-\epsilon,\theta_0+\epsilon)$ we consider the line  
	\begin{equation*}
	E(\theta)=\{(x,y)\in\R^2:(\cos\theta)(x-x_0)+(\sin\theta)(y-y_0)=0\},
	\end{equation*}
i.e., we rotate $E$ around $\vu$, and the new polygon 
\begin{equation*}
   \dual{K_\theta}=\ov E(\theta)\cap \bigcap_{i=1}^k \ov E_i.
\end{equation*}
Observe, that 
\begin{equation*}
\dual{(\dual{K_\theta})}=\conv\{\vf_\theta,\vf_1,\dots,\vf_k\} =:K_\theta 
\end{equation*} 
with 
	\begin{equation*}
	\vf_\theta=\left(\frac{\cos\theta}{\cos\theta\,x_0+\sin\theta\,y_0},\frac{\sin\theta}{\cos\theta\,x_0+\sin\theta\,y_0}\right). 
	\end{equation*} 
For $\epsilon$ we always assume that it is so small,  that the
possible  rotations  do not change the number of edges. 
Since 
	\begin{equation}
	\vf_\theta\in\{\vx\in\R^2 : \ip{\vu}{\vx}=1\}.
\label{eq:linetheta}  
	\end{equation}
the volume of $K_\theta$, as a function in  $\theta$, is monotonic in
$[\theta_0-\epsilon,\theta_0+\epsilon]$.  
	
	For each $\vv=(v_1,v_2)\notin C_0(K)$ with $v_2\ne 0$ , we have $\bnorm{\vv}_{\dual{\cs(K)}}>1$.
	Therefore, there exists $s>1$ such that
        $\bnorm{\vv}_{\dual{\cs(K)}}\geq s$ for each  $\vv\notin
        C_0(K)$ with $v_2\ne 0$. Thus, there exists
        $0<\epsilon'<\epsilon$, 
        such that for
        $\theta\in[\theta_0-\epsilon',\theta_0+\epsilon']$, it holds
        $\bnorm{\vv}_{\dual{\cs(K_\theta)}}>1$ for $\vv\notin C_0(K)$,
        $v_2\ne 0$. Since all the points $\vv\in C(K)\setminus\{\vu\}$
        are (also) contained in an edge of $\dual{K}$ different from
        $E$, we have  
        $\bnorm{\vv}_{\dual{K_\theta}}\geq  \bnorm{\vv}_{\dual{K}}$
        for 
         $|\theta-\theta_0|$ small and so 
         $\bnorm{\vv}_{\dual{\cs(K_\theta)}}\geq 1$ for all $\vv\in C_0(K)$.    
         Therefore, after a possible unimodular transformation,  
         we still have $K_\theta\in\A(t)$. Since $\vol(K_\theta)$ is
         monotonic for $|\theta-\theta_0|$ being small and
         $K\in\mathcal{M}(t)$, we conclude   $\vol(K_\theta)=\vol(K)$,
         and thus $K_\theta\in\mathcal{M}(t)$ for $|\theta-\theta_0|$ small.
	
	If $K$ is a triangle, i.e., let $\dual{K}$ has the edges
        $E,E_1,E_2$  and so $K$ has the vertices $\vf,\vf_1,\vf_2$.
        Since $\vol(K_\theta)=\vol(K)$,  \eqref{eq:linetheta} shows
        that the
        line $\{\vx\in\R^2 : \ip{\vu}{\vx}=1\}$  must be parallel to
        the edge $[\vf_1,\vf_2]$ of $K$. 

Let $\vu'\in C_0(K)$ such that
$\vu=\frac{\vu'}{\bnorm{\vu'}_{\dual{K}}}$. If $\vu'\ne\pm \ve_1$ then its
last coordinate is $1$  (cf.~Lemma \ref{lem:line} ) and hence, after an
unimodular transformation we may always assume $\vu'\in\{\pm\ve_1,\pm\ve_2\}$.


 If $\vu'\in\{\pm \ve_1\}$ then the edge $[\vf_1,\vf_2]$ has normal
 vector $\ve_1$, and in view of \eqref{eq:cs} we get that the length
 of the edge $[\vf_1,\vf_2]$ has length 2, and the height of $\vf$ with
 respect to $[\vf_1,\vf_2]$ is $2/t$. Hence, its volume is $2/t$ which
 is not minimal (cf.~\eqref{eq:upmin}) and so we are violating $K\in\mathcal{M}(t)$. 
Analogously, if $\vu'\in\{\pm \ve_2\}$  then the edge $[\vf_1,\vf_2]$ has normal
 vector $\ve_2$, and then  the length 
 of the edge $[\vf_1,\vf_2]$ is $2/t$ and the height of $\vf$ with
 respect to $[\vf_1,\vf_2]$ is $2$. Again, the  volume of the triangle is contradicting
 $K\in\mathcal{M}(t)$.

 Hence, $K$ is not a triangle, and, in particular,  all triangles in
 $\mathcal{M}(t)$ have property \eqref{eq:property}.

So let $K$ be not a triangle. By Lemma \ref{lem:translate}, we may
apply a translation to  $K$ such that the origin is contained in the
relative interior of the vertices adjacent to $\vf$. For convenience
we denotes these two vertices by $\vf_+$ and $\vf_-$, such that
$\vf_-,\vf,\vf_+$ are in clockwise order. Let  $E_-, E_{\theta_0}(=E(\theta_0)), E_+$
be the corresponding supporting lines of $\dual{K}$ and   
\begin{equation*}
	E_{+}\cap\dual{K}=[\vw_1,\vw_2],\quad E(\theta_0)\cap\dual{K}=[\vw_2,\vw_3],\quad E_{-}\cap\dual{K}=[\vw_3,\vw_4]
\end{equation*}
be the associated edges of $\dual{K}$, where $\vw_i$, $1\leq i\leq
4$, are the vertices of these edges.
\begin{center} 
\begin{figure}[hbt]
  	\begin{tikzpicture}[scale=1.5]
	\draw (0,0) node[below] {$\vnull$};
	\filldraw[black] (0,0) circle(0.4pt);
	\draw (-1,1) -- (2,1) -- (2,-1) -- (-1,-1);
	\draw (3,1/3) -- (-2,-4/3);
	\draw (0,0) -- (3,0);
	\draw (0,0) -- (3,3/2);
	\draw (-1,1) node[left] {$\vw_4$};
	\draw (0.5,1) node[above] {points of $C(K)$};
	\draw (2,1) node[right] {$\vw_3$};
	\draw (2.2,0) node[below] {$\vu$};
	\draw (2,-1) node[right] {$\vw_2$};
	\draw (0.5,-1) node[below] {points of $C(K)$};
	\draw (-1,-1) node[left] {$\vw_1$};
	\draw (2,0.5) node[right] {$E(\theta_0)$};
	\draw (1,-0.5) node[below] {$E(\theta_1)$};
	\draw (0,1.5) node {$E_{-}$};
	\draw (0,-1.5) node {$E_{+}$};
	\filldraw[black] (-1,1) circle(0.4pt);
	\filldraw[black] (0,1) circle(0.4pt);
	\filldraw[black] (0.5,1) circle(0.4pt);
	\filldraw[black] (1,1) circle(0.4pt);
	\filldraw[black] (2,1) circle(0.4pt);
	\filldraw[black] (2,0) circle(0.4pt);
	\filldraw[black] (2,-1) circle(0.4pt);
	\filldraw[black] (1,-1) circle(0.4pt);
	\filldraw[black] (0.5,-1) circle(0.4pt);
	\filldraw[black] (0,-1) circle(0.4pt);
	\filldraw[black] (-1,-1) circle(0.4pt);
	\draw (-1,0) node {$\dual{K}$};
      \end{tikzpicture}
      \caption{The non-triangle case}         
\end{figure}
\end{center} 

Since the origin $\vnull$ can only be in at most one of the triangles
$\conv\{\vu,\vw_2,\vw_1\}$ and $\conv\{\vu,\vw_3,\vw_4\}$, we assume
$\vnull\notin\conv\{\vu,\vw_2,\vw_1\}$.
Let $\theta_1\in[\theta_0-\pi,\theta_0+\pi]$ such $E(\theta_1)\cap\dual{K}=[\vu,\vw_1]$. If $\{t\vw_3:t\in\R\}\cap[\vw_1,\vw_2]=\vw'$, then let $\theta_2\in[\theta_0-\pi,\theta_0+\pi]$ such that $E(\theta_2)\cap\dual{K}=[\vu,\vw_3]$.
	
For $\vx,\vy\in\R^2$ we denote by
$\con\{\vx,\vy\}=\{\lambda\vx+\mu\vy:\,\lambda,\mu\geq 0\}$ the  
cone generated by $\vx$ and $\vy$. Now we start to rotate
$E(\theta)$ clockwise around $\vu$ and we denote the so created 
bodies by $\dual{K}_\theta$. 
Then, for each  point
$$\vx\in C_1=\con\{\vw_1,\vu\}\cap\text{cone}\{\vu,-\vw_3\}$$
its norm $\bnorm{\vx}_{\dual{K}_\theta}$ is non-decreasing and
$\bnorm{-\vx}_{\dual{K}_\theta}$ does not change; and  
 for each point
 $$\vx\in C_2=\con\{\vu,\vw_3\}\cap\con\{\vu,-\vw_1\},$$
 $\bnorm{\vx}_{\dual{K}_\theta}$ is non-increasing while
 $\bnorm{-\vx}_{\dual{K}_\theta}$ does not change. Therefore,
  \begin{equation}
    \begin{split}
      \vx\in C_1 &\Rightarrow
      \bnorm{\vx}_{\dual{\cs(K_\theta)}}\geq
      \bnorm{\vx}_{\dual{\cs(K_{\theta_0})}}\\
  \vx\in C_2 &\Rightarrow
      \bnorm{\vx}_{\dual{\cs(K_\theta)}}\leq
      \bnorm{\vx}_{\dual{\cs(K_{\theta_0})}}.   
    \end{split}
  \tag{Q}
  \label{eq:normcomp}  
  \end{equation}
Now let $\epsilon_0$ be  maximal, such that $K_\theta\in\mathcal{M}(t)$,
for all $\theta\in[\theta_0-\epsilon_0,\theta_0]$. 

If $\epsilon_0\geq\theta-\theta_1$, then
$K_{\theta_1}\in\mathcal{M}(t)$, and for each small positive number
$r$, $K_{\theta_1+r}\in\mathcal{M}(t)$. Hence, for  small enough $r$,
the corresponding edge $E_{+}\cap \dual{K_{\theta_1+r}}$ of
$\dual{K_{\theta_1+r}}$ has no point of $C(K_{\theta_1+r})$ in its
relative interior. According to Proposition  \ref{prop:one} this
contradicts  $K_{\theta_1+r}\in\mathcal{M}(t)$.

Hence, we know $\epsilon_0<\theta-\theta_1$. 
If $\vnull\in\conv\{\vw_1,\vu,\vw_3\}$ and $\epsilon_0\geq
\theta-\theta_2$, then  it still holds
$\vnull\notin\{\vu,\vw',\vw_1\}$, and we replace $K$ by $K_{\theta_2}$
and  start the rotating process again.

Hence, we may assume  $\epsilon_0<\theta-\theta_1$ and if
$\vnull\in\conv\{\vw_1,\vu,\vw_3\}$ we also may assume 
$\epsilon_0<\theta-\theta_2$.
Since $K_{\theta_0-\epsilon_0}\in\mathcal{M}(t)$ and $\epsilon_0$ is
maximal, for each small positive number $s$ we know  $K_{\theta_0-\epsilon_0-s}\notin\mathcal{M}(t)$. Then there are five cases:

\smallskip\noindent
(1) There exists a
$\vv'\in\Z^2\setminus\{\pm\ve_1,\pm\ve_2\}$ such that  
$\bnorm{\vv'}_{\dual{\cs(K_{\theta_0-\epsilon_0})}}=1$ and
$\bnorm{\vv'}_{\dual{\cs(K_{\theta_0-\epsilon_0-s})}}<1$ for some
small $s>0$. Then we have   $\vv'\in C_1\cup
C_2$ (the norms of other points are not changed).
By  \eqref{eq:normcomp}  we conclude
$\vv'\in\con\{\vu,\vw_3\}$, which means that
$\frac{\vv'}{\bnorm{\vv'}_{\dual{(K_{\theta_0-\epsilon_0})}}}$ is a
new point of $C(K_{\theta_0-\epsilon_0})$ that lies in the relative
interior of the edge $E(\theta_0-\epsilon_0)$. In this case,
$E(\theta_0-\epsilon_0)$ has two points of
$C(K_{\theta_0-\epsilon_0})$ in the relative interior and hence, $K_{\theta_0-\epsilon_0}$
fulfills  property \eqref{eq:property}.

\smallskip\noindent
(2) $\bnorm{\ve_1}_{\dual{\cs(K_{\theta_0-\epsilon_0})}}=\frac{1}{t}$
and
$\bnorm{\ve_1}_{\dual{\cs(K_{\theta_0-\epsilon_0-s})}}<\frac{1}{t}$
for some small $s>0$. Again, we may  assume  $\ve_1\in C_1\cup
C_2$ and by \eqref{eq:normcomp} $\ve_1\in\con\{\vu,\vw_3\}$. 
Then $\frac{\ve_1}{\bnorm{\ve_1}_{\dual{(K_{\theta_0-\epsilon_0})}}}$ is a
new point of $C(K_{\theta_0-\epsilon_0})$ in the relative interior of
the edge $E_{\theta_0-\epsilon_0}$. But then we have
$\bnorm{\ve_1}_{\dual{\cs(K_{\theta_0})}}>\frac{1}{t}$, implying
$\lambda_1(\dual{\cs(K_{\theta_0})})>\frac{1}{t}$, contradicting $K_{\theta_0}\in\mathcal{M}(t)$.

\smallskip\noindent
(3) $\bnorm{\ve_2}_{\dual{\cs(K_{\theta_0-\epsilon_0})}}=1$ and
$\bnorm{\ve_2}_{\dual{\cs(K_{\theta_0-\epsilon_0-s})}}<1$ for some
small $s>0$. Then  $\ve_2\in C_1\cup C_2$ and by \eqref{eq:normcomp}  
$\ve_2\in\con\{\vu,\vw_3\}$. Then 
$\frac{\ve_2}{\bnorm{\ve_2}_{\dual{(K_{\theta_0-\epsilon_0}})}}$ is a
new point of $C(K_{\theta_0-\epsilon_0})$ in the relative interior of
the edge $E_{\theta_0-\epsilon_0}$. This implies  $\bnorm{\ve_2}_{\dual{\cs(K_{\theta_0})}}>1$, contradicting $K_{\theta_0}\in\mathcal{M}(t)$.

\smallskip\noindent
(4) $\bnorm{\ve_1}_{\dual{\cs(K_{\theta_0-\epsilon_0})}}=\frac{1}{t}$
and $\bnorm{\ve_1}_{\dual{\cs(K_{\theta_0-\epsilon_0-s})}}>\frac{1}{t}$
for some small $s>0$. Then  $\ve_1\in C_1\cup C_2$ and in view of
\eqref{eq:normcomp} we get 
$\ve_1\in\con\{\vu,\vw_1\}$. The
intersection of $E_+$ and $E_{\theta_0-\epsilon_0}$ is actually
$\frac{\ve_1}{\bnorm{\ve_1}_{\dual{(K_{\theta_0-\epsilon_0})}}}$. In
this case, $E(\theta_0-\epsilon_0)$ has  $\vu$ of
$C(K_{\theta_0-\epsilon_0})$ in the relative interior and
$\frac{\ve_1}{\bnorm{\ve_1}_{\dual{K_{\theta_0-\epsilon_0}}}}$ as a
vertex. Hence,  $K_{\theta_0-\epsilon_0}$ satisfies property \eqref{eq:property}.

\smallskip\noindent
(5) There exists a
$\vv'\in\Z^2\setminus\{\pm\ve_1\}$ such that  
$\bnorm{\vv'}_{\dual{\cs(K_{\theta_0-\epsilon_0})}}=1$ and
$\bnorm{\vv'}_{\dual{\cs(K_{\theta_0-\epsilon_0-s})}}>1$ for some
small $s>0$. Then $\vv'\in\con\{\vu,\vw_1\}$ (cf. \eqref{eq:normcomp}), and the
intersection of $E_+$ and $E_{\theta_0-\epsilon_0}$ is actually the
point $\frac{\vv'}{\bnorm{\vv'}_{\dual{(K_{\theta_0-\epsilon_0})}}}$.
If $\vv'\neq\pm\ve_2$ then rotation would not stop here. Hence, let
$\vv'=\ve_2$. Since $K_{\theta_0-\epsilon_0}$ has at least $4$ edges,
and the relative interior of each edge of
$\dual{K_{\theta_0-\epsilon_0}}$ contains a point of
$C(K_{\theta_0-\epsilon_0})$ (cf. Proposition \ref{prop:one}) , and since now  $\ve_2$  is also a vertex
of $C(K_{\theta_0-\epsilon_0})$ we find
$|C(K_{\theta_0-\epsilon_0})|=6$ (cf.~Remark \ref{rem:six}). Hence, there exists a unimodular transformation  
of $K_{\theta_0-\epsilon_0}$ mapping $\ve_2$ to a point
$C(K_{\theta_0-\epsilon_0})\setminus\{\pm\ve_1,\pm\ve_2\}$ and we
start the rotating process with this new body.

\end{proof}

Next we exclude the quadrilateral case. 

\begin{prop}
  There are no quadrilaterals in $\mathcal{M}(t)$.
	\label{prop:noquad}
\end{prop}

\begin{proof}
  Let $K$  be a quadrilateral in $\mathcal{M}(t)$. According to the
  proof of Proposition \ref{prop:a2} we may assume that $K$ satisfies
  property \eqref{eq:property}. 

Together with  Proposition \ref{prop:six} we conclude that
$\frac{\ve_1}{\bnorm{\ve_1}_{\dual{K}}}$ and
$\frac{-\ve_1}{\bnorm{-\ve_1}_{\dual{K}}}$ are two opposite vertices
of $\dual{K}$ and each of the four edges of $\dual{K}$ has a point of
$C(K)$ in the relative interior. In view of  Remark \ref{rem:six},  we
may assume with $\vu_1=\ve_1+\ve_2$ that $C_0(K)=\{\pm\ve_1, \pm\ve_2,
\pm\vu_1\}$.
	
	We translate $K$ into a position, such that
        $\bnorm{\vu_1}_{\dual{K}}=\bnorm{-\vu_1}_{\dual{K}}=1$ and
        $\bnorm{\ve_2}_{\dual{K}}=\bnorm{-\ve_2}_{\dual{K}}=1$.
        In order to do so, we first find the four supporting
        hyperplanes of $K$ with normal vectors $\pm\vu_1,\pm\ve_2$ and
        find the center of this parallelogram. The center of this
        parallelogram is in the interior of $K$, and thus we can
        translate the origin point to the center of this parallelogram.
	
	Let  $t_1=\bnorm{\ve_1}_{\dual{K}}$ and
        $t_2=\bnorm{-\ve_1}_{\dual{K}}$. Then  $t_1+t_2=\frac{2}{t}$.
	
	\begin{center}
	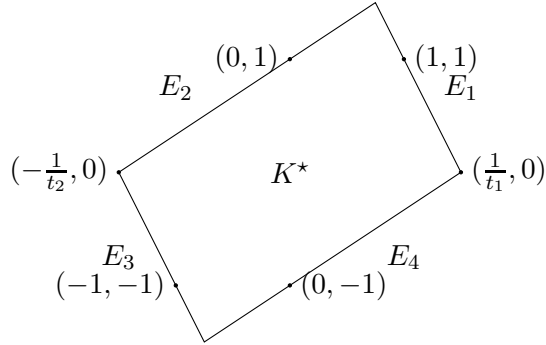
\begin{figure}[hbt]
	\begin{tikzpicture}[scale=1.5]
	\draw (3/2,0) -- (3/4,3/2) -- (-3/2,0) -- (-3/4,-3/2) -- (3/2,0);
	\draw (3/2,0) node[right] {$(\frac{1}{t_1},0)$};
	\draw (1,1) node[right] {$(1,1)$};
	\draw (0,1) node[left] {$(0,1)$};
	\draw (-3/2,0) node[left] {$(-\frac{1}{t_2},0)$};
	\draw (-1,-1) node[left] {$(-1,-1)$};
	\draw (0,-1) node[right] {$(0,-1)$};
	\filldraw[black] (3/2,0) circle(0.4pt);
	\filldraw[black] (1,1) circle(0.4pt);
	\filldraw[black] (0,1) circle(0.4pt);
	\filldraw[black] (-3/2,0) circle(0.4pt);
	\filldraw[black] (-1,-1) circle(0.4pt);
	\filldraw[black] (0,-1) circle(0.4pt);
	\draw (3/2,3/4) node {$E_1$};
	\draw (-1,3/4) node {$E_2$};
	\draw (-3/2,-3/4) node {$E_3$};
	\draw (1,-3/4) node {$E_4$};
	\draw (0,0) node {$\dual{K}$};
	\end{tikzpicture}
	\caption{The polar body of a quadrilateral satisfying the condition \eqref{eq:property}}
	\end{figure}
	\end{center}

        Next we consider all the linear equations describing
        the edges of $\dual{K}$ and so we get the vertices of $K$.  
        \begin{enumerate}
         \item  The affine hull of the edge of $\dual{K}$ containing $(\frac{1}{t_1},0)$
           and $\ve_1+\ve_2$ is given by the equation
           $\{(x,y) : t_1x+(1-t_1)y=1\}$, 
           and so $K$ has the vertex $(t_1,1-t_1)$.
\item  The affine hull of the edge of $\dual{K}$ containing $\ve_2$ and $(-\frac{1}{t_2},0)$ is given by the equation
           $\{(x,y) : 	-t_2x+y=1\}$, 
           and so $K$ has the vertex $(-t_2,1)$.
\item  The affine hull of the edge of $\dual{K}$ containing
  $(-\frac{1}{t_2},0)$ and $-(\ve_1+\ve_2)$ $\ve_2$  is given by the equation
           $\{(x,y) : 	-t_2x-(1-t_2)y=1	\}$, 
           and so $K$ has the vertex $(-t_2,-1+t_2)$.
\item  The affine hull of the edge of $\dual{K}$ containing
  $-\ve_2$ and $(\frac{1}{t_1},0)$  is given by the equation
           $\{(x,y) : 	t_1x-y=1	\}$, 
           and so $K$ has the vertex $(t_1,-1)$.           
          \end{enumerate} 
  Therefore $\vol(K)=\frac{4}{t}-\frac{2}{t^2}$, and hence  $K\notin\mathcal{M}(t)$ (cf. \eqref{eq:upmin}).
	
	\begin{center}
	\begin{figure}[hbt]
	\begin{tikzpicture}[scale=1.5]
	\draw (2,-1) -- (0,1) -- (-2,1) -- (0,-1) -- (2,-1);
	\draw (0,0) node {$K$};
	\draw (2/3,1/3) -- (-2/3,1) -- (-2/3,-1/3) -- (2/3,-1) -- (2/3,1/3);
	\draw (2/3,1/3) node[right] {$(t_1,1-t_1)$};
	\draw (-2/3,1) node[above] {$(-t_2,1)$};
	\draw (-2/3,-1/3) node[left] {$(-t_2,-1+t_2)$};
	\draw (2/3,-1) node[below] {$(t_1,-1)$};
	\filldraw[black] (2/3,1/3) circle(0.4pt);
	\filldraw[black] (-2/3,1) circle(0.4pt);
	\filldraw[black] (-2/3,-1/3) circle(0.4pt);
	\filldraw[black] (2/3,-1) circle(0.4pt);
	\end{tikzpicture}
	\caption{A quadrilateral satisfying the condition \eqref{eq:property}}
	\end{figure}
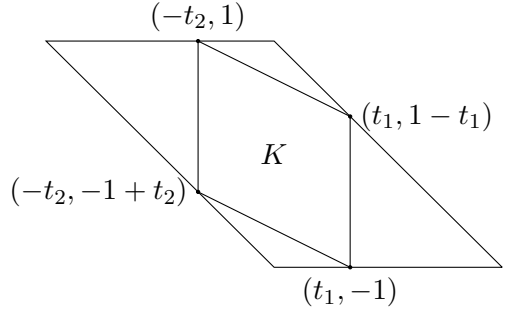
	\end{center}
	
\end{proof}


Finally, we consider the triangles in $\mathcal{M}(t)$.

\begin{prop} Up to
translations and unimodular transformations,  $\mathcal{M}(t)$ contains
only the triangle $T_{1,1/t} =\conv\{(-1/t,1-1/t), (1/t,1), (0,-1)\}$
of volume $\frac{2}{t}-\frac{1}{2}\frac{1}{t^2}$. 
	\label{prop:min}
\end{prop}

\begin{proof} 
  Let $K\in\mathcal{M}(t)$. According to Proposition \ref{prop:a2} and
  Proposition \ref{prop:noquad}, $K$ is a triangle 
   satisfying property \eqref{eq:property}.  Thus we know  
  \begin{enumerate}
     \item[1.] only one edge of $\dual{K}$ contains two points of $C(K)$
     while the other
     two edges share a vertex in $C(K)$ and separately have one point
     of $C(K)$ in the relative interior of each edge, or 
   \item[2.]  each edge of $\dual{K}$ contains two points of $C(K)$ in the
   relative interior. 
   \end{enumerate}
	Therefore $|C(K)|$ has to be $6$ (cf.~Proposition \ref{prop:six}). 
	According to Remark \ref{rem:six}, we may assume that up to an
        unimodular transformation $C_0(K)=\{\pm\ve_1, \pm \ve_2,
        \pm(\ve_2-\ve_1)\}$.
        
        Next we discuss the above two different cases.

\smallskip\noindent
1. Here we may assume that $\frac{\ve_1}{\bnorm{\ve_1}_{\dual{K}}}$ is
a vertex of $\dual{K}$.  Then
$\frac{-\ve_1}{\bnorm{-\ve_1}_{\dual{K}}}$ has to be in the edge
opposite to this vertex. Since the two edges of $\dual{K}$ sharing
the vertex $\frac{\ve_1}{\bnorm{\ve_1}_{\dual{K}}}$ must contain
$\frac{\ve_2}{\bnorm{\ve_2}_{\dual{K}}}$ and
$\frac{\ve_1-\ve_2}{\bnorm{\ve_1-\ve_2}_{\dual{K}}}$ in their relative
interior, respectively, 
the remaining edge  contains either
 both  points $\frac{-\ve_2}{\bnorm{-\ve_2}_{\dual{K}}}$,  $\frac{\ve_2-\ve_1}{\bnorm{\ve_2-\ve_1}_{\dual{K}}}$, or only  one of
 these points. Here we just consider the case that this edge contains both
 points, because otherwise each edge contains  two points of $C(K)$
 and this   will be discussed in the next case.

 Since $\bnorm{\ve_2}_{\dual{K}}+\bnorm{-\ve_2}_{\dual{K}}=2$ and
 $\bnorm{\ve_2-\ve_1}_{\dual{K}}+\bnorm{\ve_1-\ve_2}_{\dual{K}}=2$, we
 choose a translation of $K$, such that
 $\bnorm{\ve_2}_{\dual{K}}=\bnorm{-\ve_2}_{\dual{K}}=1$ and
 $\bnorm{\ve_2-\ve_1}_{\dual{K}}=\bnorm{\ve_1-\ve_2}_{\dual{K}}=1$. In
 order to do so, we first find the four supporting hyperplanes of $K$
 with normal vectors $\pm\ve_2,\pm(\ve_2-\ve_1)$ and find the center
 of this parallelogram. Then the  center of this parallelogram has to
 be an interior point  of $K$, and we can translate the origin  to
 this center. 
	
	Then one edge of $\dual{K}$ contains
        $\frac{\ve_1}{\bnorm{\ve_1}_{\dual{K}}}$ and $\ve_2$, one edge
        contains $\ve_2-\ve_1$ and $-\ve_2$, and one edge contains
        $\ve_1-\ve_2$ and
        $\frac{\ve_1}{\bnorm{\ve_1}_{\dual{K}}}$. From this we get
        $\bnorm{-\ve_1}_{\dual{K}}=2$ and thus $\bnorm{\ve_1}_{\dual{K}}=2\bnorm{\ve_1}_{\dual{\cs(K)}}-\bnorm{-\ve_1}_{\dual{K}}=\frac{2}{t}-2<0$, which is impossible.
	
	\begin{center}
	\begin{figure}[hbt]
	\begin{tikzpicture}[scale=1.5]
	\draw (3/2,0) -- (-3/2,2) -- (1/2,-2) -- (3/2,0);
	\draw (0,1) node[right] {$(0,1)$};
	\draw (-1,1) node[left] {$(-1,1)$};
	\draw (0,-1) node[left] {$(0,-1)$};
	\draw (1,-1) node[right] {$(1,-1)$};
	\draw (3/2,0) node[above] {$\frac{\ve_1}{\bnorm{\ve_1}_{\dual{K}}}$};
	\draw (-1/2,0) node[below] {$(-1/2,0)$};
	\filldraw[black] (0,1) circle(0.4pt);
	\filldraw[black] (-1,1) circle(0.4pt);
	\filldraw[black] (0,-1) circle(0.4pt);
	\filldraw[black] (1,-1) circle(0.4pt);
	\filldraw[black] (-1/2,0) circle(0.4pt);
	\filldraw[black] (3/2,0) circle(0.4pt);
	\draw (0,1/3) node {$\dual{K}$};
	\draw (2,0) -- (-2,0);
	\end{tikzpicture}
	\caption{First (impossible) triangle case satisfying the condition \eqref{eq:property}}
	\end{figure}
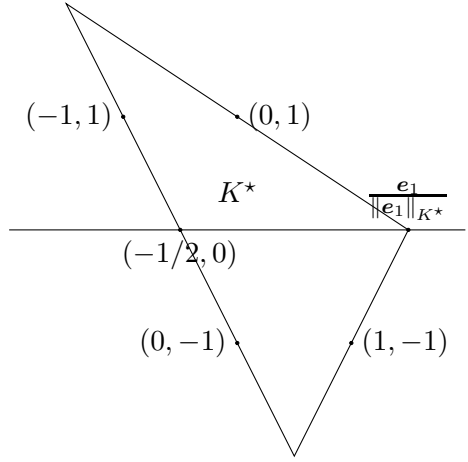
	\end{center}
	
Hence, it remains only to consider the second case, and here we just assume
that

\smallskip\noindent
2.  each edge of $\dual{K}$ contains  two points of $C(K)$. Up to
a rotation by $\pi$, i.e., up to an unimodular transformation,  we may assume that the three edges $U_i$, $1\leq
i\leq 3$,  are given 
	\begin{equation}
	\begin{split}
	& \text{ $U_1$ contains $\frac{\ve_1}{\bnorm{\ve_1}_{\dual{K}}}$ and $\frac{\ve_2}{\bnorm{\ve_2}_{\dual{K}}}$,}\\
	& \text{ $U_2$ contains $\frac{\ve_2-\ve_1}{\bnorm{\ve_2-\ve_1}_{\dual{K}}}$ and $\frac{-\ve_1}{\bnorm{-\ve_1}_{\dual{K}}}$,}\\
	& \text{ $U_3$ contains $\frac{-\ve_2}{\bnorm{-\ve_2}_{\dual{K}}}$ and $\frac{\ve_1-\ve_2}{\bnorm{\ve_1-\ve_2}_{\dual{K}}}$.}
	\end{split}
	\end{equation}
	
	\begin{center}
	\begin{figure}[hbt]
	\begin{tikzpicture}[scale=1.5]
	\draw (2,-1) -- (-1,2) -- (-1,-1) -- (2,-1);
	\filldraw[black] (1,0) circle(0.4pt);
	\filldraw[black] (0,1) circle(0.4pt);
	\filldraw[black] (-1,1) circle(0.4pt);
	\filldraw[black] (-1,0) circle(0.4pt);
	\filldraw[black] (0,-1) circle(0.4pt);
	\filldraw[black] (1,-1) circle(0.4pt);
	\draw (1,0) node[right] {$(\frac{1}{t_1},0)$};
	\draw (0,1) node[above] {$(0,1)$};
	\draw (-1,1) node[left] {$(-1,1)$};
	\draw (-1,0) node[left] {$(-\frac{1}{t_2},0)$};
	\draw (0,-1) node[below] {$(0,-1)$};
	\draw (1,-1) node[below] {$(1,-1)$};
	\draw (1,1) node {$U_1$};
	\draw (-3/2,1/2) node {$U_2$};
	\draw (1/2,-3/2) node {$U_3$};
	\draw (0,0) node {$\dual{K}$};
	\end{tikzpicture}
	\caption{Second triangle case satisfying the condition \eqref{eq:property}}
	\end{figure}
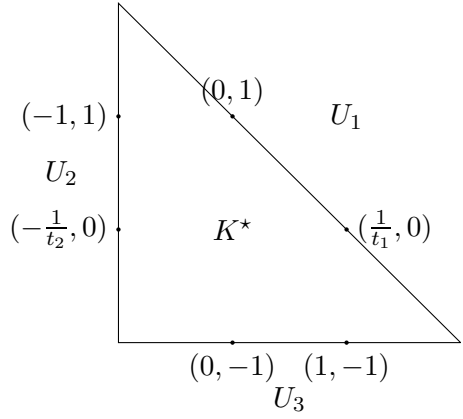
	\end{center}

	Since $\bnorm{\ve_2}_{\dual{K}}+\bnorm{-\ve_2}_{\dual{K}}=2$
        and
        $\bnorm{\ve_2-\ve_1}_{\dual{K}}+\bnorm{\ve_1-\ve_2}_{\dual{K}}=2$,
        we choose a translation of $K$, such that
        $\bnorm{\ve_2}_{\dual{K}}=\bnorm{-\ve_2}_{\dual{K}}=1$ and
        $\bnorm{\ve_2-\ve_1}_{\dual{K}}=\bnorm{\ve_1-\ve_2}_{\dual{K}}=1$. In
        order to do so, we proceed as in case 1., i.e., first we find
        the four supporting hyperplanes of $K$ with normal vectors
        $\pm\ve_2,\pm(\ve_2-\ve_1)$ and find the center of this
        parallelogram. Then the  center of this parallelogram has to
 be an interior point  of $K$, and we can translate the origin  to
 this center.

 Let 
	\begin{equation*}
	\begin{split}
	t_1=\bnorm{\ve_1}_{\dual{K}}, \,  
	t_2=\bnorm{-\ve_1}_{\dual{K}}.
	\end{split}
	\end{equation*}
	Since $\bnorm{\ve_1}_{\dual{\cs(K)}}=\frac{1}{t}$, we have
	\begin{equation}
          t_1+t_2=\frac{2}{t}.
          \label{eq:last} 
      \end{equation}
      
	The affine hull of the edge $U_1$ containing
        $(\frac{1}{t_1},0)$ and $\ve_2$ is given by 
	\begin{equation*}
	\left\{(x,y) : 	t_1x+y=1\right\}.
      \end{equation*}
      	The affine hull of the edge $U_2$ containing $\ve_2-\ve_1$ and
        $(-\frac{1}{t_2},0)$ is given by 
	\begin{equation*}
	\left\{(x,y) : 	-t_2x+(1-t_2)y=1\right\}.
	\end{equation*}
	The affine hull of the edge $U_3$ containing $-\ve_2$ and
        $\ve_1-\ve_2$ is given by 
	\begin{equation*}
	\left\{(x,y) : 	-y=1\right\}.
	\end{equation*}
	Therefore,
	\begin{equation*}
	K=\conv\{(t_1,1), (-t_2,1-t_2), (0,-1)\},
	\end{equation*}
	with volume 
	\begin{equation*}
	\begin{split}
	\vol(K)=&\frac{1}{2}(t_1(1-t_2)+t_2)+\frac{1}{2}t_2+\frac{1}{2}t_1\\
	=&\frac{2}{t}-\frac{1}{2}t_1t_2\\
	\geq&\frac{2}{t}-\frac{1}{2}\frac{1}{t^2}.
	\end{split}
      \end{equation*}
       In the last inequality we have used \eqref{eq:last} and the
       arithmetic-geometric mean inequality. Hence, we have equality
       if and only if 
	\begin{equation}
	t_1=t_2=\frac{1}{t}.
	\label{eq:final}
	\end{equation}
	Since $K$ is supposed to have minimal volume
        (cf. \eqref{eq:upmin}) we have equality.
	
	
	Therefore, in this case, $K$ is a translation of $\conv\{(\frac{1}{t},1),(-\frac{1}{t},1-\frac{1}{t}),(0,-1)\}$. 
	
\end{proof}


\begin{thebibliography}{}
	
	
	\bibitem{ABT16}
	{\sc J.C.\'Alvarez Paiva, F.Balacheff, K.Tzanev},
	{\it Isosystolic inequalities for optical hypersurfaces},
	Advances in Mathematics, {\bf 301}, (2016), 934--972.
	
	\bibitem{Eg61}
	{\sc H.G.Eggleston},
	{\it Note on a conjecture of L.Santal\`o},
	Mathematika, {\bf 8}, (1961), 63--65.
	
	
	
	
	\bibitem{FeM74}
	{\sc G.Fejes T\'oth, E.Makai Jr.},
	{\it On the thinnest non-separable lattice of convex plates},
	Stud. Sci. Math. Hung., {\bf 9}, (1974), 191--193.

      \bibitem{GSch17}
        {\sc Bernardo Gonz\'{a}lez Merino, Matthias  Schymura}, 
        {\it On densities of lattice arrangements intersecting every
          {$i$}-dimensional affine subspace},
        Discrete Comput. Geom., {\bf 58}(3), (2017), 663--685.

        
	\bibitem{Gr07}
	{\sc P.M.Gruber},
	{\it Convex and Discrete Geometry}.
	Springer, Berlin Heidelberg, 2007.
	
	\bibitem{GrL87}
	{\sc P.M.Gruber, C.G.Lekkerkerker},
	{\it Geometry of Numbers}.
	North Holland, Amsterdam, 1987.
	
       \bibitem{Gruenbaum:1960}
        {\sc  Branko Gr\"unbaum}, 
        {\it Partitions of mass-distributions and of convex bodies by
          hyperplanes}, 
        Pacific J. Math., {\bf 10}, (1960), 1257--1261.


	\bibitem{HHC16}
	{\sc M.Henk, M.Henze, M.A.H.Cifre},
	{\it On extensions of minkowski's theorem on successive minima},
	Forum Math., {\bf 28(2)}, (2016), 311--326.
	
	\bibitem{IS17}
	{\sc H.Iriyeh, M.Shibata},
	{\it Symmetric Mahler's conjecture for the volume product in the three dimensional case},
	arXiv:1706.01749v2.
	
        \bibitem{Ku08}
        {\sc G. Kuperberg},
        {\it From the {M}ahler conjecture to {G}auss linking integrals},
        {Geom. Funct. Anal.}, {\bf 18}(3), (2008), 870--892.

	\bibitem{Mak78}
	{\sc E.Makai Jr.},
	{\it On the thinnest non-separable lattice of convex bodies},
	Stud. Sci. Math. Hung., {\bf 13}, (1978), 19--27.
	

        \bibitem{MM16}
	{\sc E.Makai Jr., H.Martini},
	{\it Density estimates for {$k$}-impassable lattices of balls and general convex bodies in $\R^n$},
	arXiv:1612.01307v1.
	

	\bibitem{Mi96}
	{\sc H.Minkowski},
	{\it Geometrie der Zahlen}.
	Teubner, Leipzig-Berlin, 1896, reprinted by Johnson Reprint Corp., New York, 1968.
	
	
	\bibitem{Mah39}
	{\sc K.Mahler},
	{\it Ein Minimalproblem f\"ur konvexe Polygone},
	Mathematica (Zutphen), {\bf B 7}, (1939), 118--127.

	\bibitem{Mah39-2}
	{\sc K.Mahler},
	{\it Ein \"Ubertragungsprinzip f\"ur konvexe K\"orper},
	\v Casopis P\v est, {\bf 68}, (1939), 93--102.

        \bibitem{Mah74}
         {\sc K. Mahler} 
         {\it Polar analogues of two theorems by Minkowski},
         Bull. Austral. Math. Soc., {\bf 11}, 1974, 121--129.

	
	
	
	
\bibitem{Schneider:2014}
	{\sc Rolf Schneider}, 
	{\it Convex bodies: the {B}runn-{M}inkowski theory}, volume 151 of
	{\em Encyclopedia of Mathematics and its Applications},
        Cambridge University Press, 
        Cambridge, expanded edition, 2014.

	
\end{thebibliography}
\end{document}